\documentclass{amsart}   
\usepackage{amsmath}
\usepackage[mathscr]{eucal} 
\usepackage{amssymb}
\usepackage{latexsym}
\usepackage{amsthm} 
\theoremstyle{plain}
\newtheorem{theorem}{Theorem}

\newtheorem{theorema}{Theorem A\!\!}
\newtheorem{proposition}{Proposition}

\newtheorem{corollaryz}{Corollary\!\!}
\newtheorem{lemma}{Lemma}

\numberwithin{equation}{section}  
\theoremstyle{definition}

\theoremstyle{remark}
 
\newtheorem{remarkz}{Remark\!\!} 
 
\title[singular integrals along surfaces of revolution ]{Estimates for  singular integrals along surfaces of revolution  }  
 \author{Shuichi Sato} 

\begin{document}
\setcounter{page}{1}
\address{Department of Mathematics, 
Faculty of Education, 
Kanazawa University,    
Kanazawa 920-1192, 
Japan} 
\email{shuichi@kenroku.kanazawa-u.ac.jp} 
\thanks{
\endgraf 
}

\maketitle 

\begin{abstract} 
We prove certain $L^p$ estimates ($1<p<\infty$) for non-isotropic 
singular integrals 
along surfaces of revolution. The singular integrals are defined 
by rough kernels.  As an application we obtain 
$L^p$ boundedness of the singular integrals under a sharp size condition on 
their kernels.  
 We also prove a certain estimate for a trigonometric integral, which 
 is useful in studying non-isotropic singular integrals. 
\end{abstract}

\section{ Introduction}  
Let $P$ be an $n\times n$ real matrix whose eigenvalues have positive real 
parts.  Let $\gamma=\text{{\rm trace} $P$}$. Define a dilation group 
$\{A_t\}_{t>0}$ on $\Bbb R^n$ by $A_t=t^P=\exp((\log t) P)$.   
We assume $n\geq 2$.  
There is a non-negative 
function $r$ on $\Bbb R^n$ associated with $\{A_t\}_{t>0}$.  The function $r$ 
 is continuous on $\Bbb R^n$ and infinitely differentiable in 
 $\Bbb R^n\setminus \{0\}$;  furthermore it satisfies 
\begin{enumerate}  
\renewcommand{\labelenumi}{(\arabic{enumi})}  
\item $r(A_tx)=tr(x)$ for all $t>0$ and $x\in \Bbb R^n$;  
\item $r(x+y)\leq C(r(x)+r(y))$ for some $C>0$; 
\item  if $\Sigma=\{x\in \Bbb R^n: r(x)=1\}$, then 
$\Sigma=\{\theta\in \Bbb R^n: \langle B\theta, \theta\rangle=1\}$ 
for a positive symmetric matrix $B$, where $\langle\cdot,\cdot\rangle$ 
denotes the inner product in $\Bbb R^n$.   
 \end{enumerate}    
 Also, we have $dx=t^{\gamma-1}\ d\sigma\,dt$, that is,  
$$\int_{\Bbb R^n}f(x)\,dx=\int_0^\infty\int_\Sigma f(A_t\theta)t^{\gamma-1}\,
d\sigma(\theta)\,dt $$ 
for appropriate functions $f$, where $d\sigma$ is a $C^\infty$ measure on 
$\Sigma$.  
See \cite{CT, R, SW2} for more details.  
\par 
Let $\Omega$ be locally integrable in $\Bbb R^n\setminus\{0\}$ and 
homogeneous of degree $0$ with respect to the dilation group 
$\{A_t\}$, that is,  $\Omega(A_tx)=\Omega(x)$ for $x\neq 0$. 
We assume  that 
$$ \int_\Sigma \Omega(\theta)\, d\sigma(\theta)=0.$$ 
 For $s\geq 1$, let $\Delta_s$ denote the collection of measurable functions 
$h$ on $\Bbb R_+=\{t\in \Bbb R : t>0\}$ satisfying 
$$\|h\|_{\Delta_s}=
\sup_{j\in \Bbb Z}\left(\int_{2^j}^{2^{j+1}}|h(t)|^s\,dt/t\right)^{1/s}
<\infty,$$  
where $\Bbb Z$ denotes the set of integers. 
We define $\|h\|_{\Delta_\infty}$ as usual ($\|h\|_{\Delta_\infty}=
\|h\|_{L^\infty(\Bbb R_+)}$).   
\par 
 Let $\Gamma:[0, \infty)\to \Bbb R^m$ be 
 a continuous mapping satisfying $\Gamma(0)=0$.   
We define 
a singular integral operator along the surface $(y,\Gamma(r(y)))$ by 
\begin{align} 
 Tf(x,z)
 &=\mathop{\mathrm{p.v.}}\int_{\Bbb R^n}f(x-y, z-\Gamma(r(y)))K(y)\,dy   
 \\ 
 &=\lim_{\epsilon\to 0}\int_{r(y)>\epsilon}f(x-y, z-\Gamma(r(y)))K(y)\,dy, 
 \notag  
   \end{align}  
where $K(y)=h(r(y))\Omega(y')r(y)^{-\gamma}$,  $y'=A_{r(y)^{-1}}y$ and 
 $h\in \Delta_1$.   We assume that  
the principal value integral in (1.1) exists for every 
$(x,z)\in \Bbb R^n\times \Bbb R^m$ and $f\in \mathscr S(\Bbb R^n\times \Bbb R^m)$ (the Schwartz class).  
\par 
 We denote by $L\log L(\Sigma)$  the Zygmund class of all those 
 functions  $\Omega$ on $\Sigma$ 
  which satisfy 
  $$\int_{\Sigma}|\Omega(\theta)|\log(2+|\Omega(\theta)|)\, d\sigma(\theta)
  <\infty.$$  
  Also, we consider the $L^q(\Sigma)$ spaces and write $\|\Omega\|_q=
  \left(\int_{\Sigma}|\Omega(\theta)|^q\,d\sigma(\theta)\right)^{1/q}$ for 
  $\Omega\in  L^q(\Sigma)$ ($\|\Omega\|_\infty$ is defined as usual).    
  \par 
 Let 
 $$ M_\Gamma g(z)=\sup_{R>0}R^{-1}\int_0^R|g(z-\Gamma(t))|\,dt. $$
 We assume that the maximal operator $M_\Gamma$ is bounded on $L^p(\Bbb R^m)$ 
 for all $p>1$.  See \cite{St, SW2} for examples of such functions 
 $\Gamma$. 
 \par 
  In this note we  prove the following.   
 \begin{theorem} Let $T$ be as in $(1.1)$. Suppose that  
$\Omega\in L^q(\Sigma)$ for some $q\in (1, 2]$ and 
$h\in \Delta_s$ for some $s>1$.  Then,  we have 
$$\|Tf\|_{L^p(\Bbb R^{n+m})} \leq C_p(q - 1)^{-1}
\|\Omega\|_q\|h\|_{\Delta_s}\|f\|_{L^p(\Bbb R^{n+m})} $$  
if   $|1/p- 1/2|<\min(1/s',1/2)$, 
where $1/s'+1/s=1$ and  the constant $C_p$ is independent of $q$ and    
$\Omega$. 
\end{theorem}  

\begin{theorem}  Suppose $\Omega \in L\log L(\Sigma)$ and $h\in \Delta_s$ 
for some $s>1$.  Then,  $T$ is bounded on $L^p(\Bbb R^{n+m})$ 
if $|1/p- 1/2|<\min(1/s',1/2)$.  
\end{theorem}    
Theorem 2 follows from Theorem 1 by an extrapolation method.  
When $r(x)=|x|$ (the Euclid norm), $m=1$ and $\Gamma$ is a $C^2$, convex, 
increasing function, Theorem 2 was proved in A.~Al-Salman and Y.~Pan 
\cite{A-SP} (see \cite[Theorem 4.1]{A-SP} and also \cite{KWWZ} for a related 
result).  
In \cite{A-SP}, it is noted that  the 
estimates as $q\to 1$ of Theorem 1 (in their setting) can be used through 
extrapolation to prove the $L^p$ boundedness of \cite[Theorem 4.1]{A-SP}, 
although such estimates are yet to be proved.    In this note, we are able to 
prove Theorem 1 and apply it to prove Theorem 2.  
\par 
If $\Gamma\equiv 0$ ($\Gamma$ is identically $0$), 
then $T$ essentially  reduces to the lower dimensional 
 singular integral 
 \begin{equation} 
  Sf(x)=\mathrm{p. v.} \int_{\Bbb R^n}f(x-y)K(y)\,dy.   
  \end{equation}  
For this singular integral we have the following.  
 \begin{theorem} Let 
$\Omega\in L^q(\Sigma)$ and $h\in \Delta_s$ for some $q, s \in (1, 2]$. 
Then we have 
$$\|Sf\|_{L^p(\Bbb R^{n})} \leq C_p(q - 1)^{-1}(s - 1)^{-1}
\|\Omega\|_q\|h\|_{\Delta_s}\|f\|_{L^p(\Bbb R^{n})} $$  
for all $p \in (1,\infty)$, where the constant $C_p$ is independent of 
$q, s, \Omega$ and $h$. 
\end{theorem}   
For $a > 0$, let
$$ 
L_a(h) = \sup_{j \in \Bbb Z} 
\int_{2^j}^{2^{j+1}}|h(r)|\left(\log(2+|h(r)|)\right)^a\, dr/r.  
$$  
We define a class $\mathscr L_a$  to be the space of all those 
measurable functions $h$ on $\Bbb R_+$ which satisfy $L_a(h)<\infty$.  
\par 
By Theorem 3 and an extrapolation we have the following.  
\begin{theorem}  Suppose $\Omega \in L\log L(\Sigma)$ and  
$h\in \mathscr L_a$ for some $a>2$.  Then $S$ is bounded on $L^p(\Bbb R^n)$ 
for all $p\in (1,\infty)$.  
\end{theorem} 
\par  
It is noted in \cite{DR} that $S$ is bounded on $L^p$ , $1<p<\infty$, if 
$\Omega\in L^q$ for some $q>1$ and $h\in \Delta_2$ 
(see \cite[Corollary 4.5]{DR}).  Theorem 4 improves that result. 
 See \cite{R, SW} for non-isotropic singular integrals $S$ with $h\equiv 1$ and also \cite{CZ, F, H, N} for related results. 
 \par 
 In Section 2, we prove Theorems 1 and 3. The proofs are based on the method 
 of \cite{DR}.   As in \cite{S}, a key idea of the 
 proof of Theorem 1 is to use a Littlewood--Paley decomposition depending on $q$ for which $\Omega \in L^q$. Theorem 3 is proved in a similar fashion.  
 Applying an extrapolation argument,  we can prove Theorems 2 and 4 
 from Theorems 1 and 3, respectively. We give a proof of Theorem 4 in 
   Section 3. 
 In Section 4, we prove an estimate for a trigonometric integral,
 a corollary of which is used in proving  Theorems 1 and 3.   
\par 
    Throughout this note,  
 the letter $C$ will be used to denote non-negative constants which may be 
different in different occurrences.

\section{Proofs of Theorems 1 and 3}    

Let $A^*$ denote the adjoint of a matrix $A$. 
Then $A_t^*=\exp((\log t)P^*)$. We write $A_t^*=B_t$. 
We can define a non-negative function $s$ from $\{B_t\}$ exactly 
in the same way as we define $r$ from $\{A_t\}$.   
\par
There are positive constants $c_1, c_2, c_3, c_4, \alpha_1, \alpha_2, 
\beta_1$ and $\beta_2$ such that 
\begin{gather*} 
c_1|x|^{\alpha_1}<r(x)<c_2|x|^{\alpha_2} \quad \text{if $r(x)\geq 1$}, 
\\ 
c_3|x|^{\beta_1}<r(x)<c_4|x|^{\beta_2} \quad \text{if $0<r(x)\leq  1$}.  
\end{gather*}    
Also, we have 
\begin{gather*}  
d_1|\xi|^{a_1}<s(\xi)<d_2|\xi|^{a_2} \quad \text{if $s(\xi)\geq 1$}, 
\\ 
d_3|\xi|^{b_1}<s(\xi)<d_4|\xi|^{b_2} \quad \text{if $0<s(\xi)\leq 1$} 
\end{gather*}    
for some positive numbers $d_1, d_2, d_3, d_4$, $a_1$, $a_2$, $b_1$ and 
 $b_2$ (see \cite{SW2}).  These estimates are useful in the following.  
\par  
We consider the singular integral operator  $T$ defined in (1.1).  
Let  $E_j=\{x\in \Bbb R^n: \beta^j<r(x)\leq \beta^{j+1}\}$, where 
$\beta\geq 2$ and $j\in \Bbb Z$. We define a sequence of Borel measures 
$\{\sigma_j\}$ on $\Bbb R^n\times\Bbb R^m$ by 
$$\hat{\sigma}_j(\xi,\eta)=\int_{E_j}e^{-2\pi i\langle y,\xi\rangle} 
e^{-2\pi i\langle\Gamma(r(y)),\eta\rangle} K(y)\, dy,$$ 
where  $\hat{\sigma}_j$ denotes the Fourier transform of $\sigma_j$ defined by 
 $$\hat{\sigma}_j(\xi,\eta) 
= \int e^{-2\pi i\langle (x,z),(\xi,\eta)\rangle}\,d\sigma_j(x,z).$$  
  Then $Tf(x) = \sum_{-\infty}^{\infty}\sigma_k * f(x)$. 
\par 
Let $\mu_k = |\sigma_k|$, where $|\sigma_k|$ denotes the total variation of 
$\sigma_k$. Let $\Omega\in L^q$, $h\in \Delta_s$, $q, s\in (1, 2]$. 
 We prove the following estimates (2.1)--(2.5): 
\begin{equation} 
\|\sigma_k \|  \leq C(\log \beta)\|\Omega\|_1\|h\|_{\Delta_1}\leq 
C(\log \beta)\|\Omega\|_q\|h\|_{\Delta_s}, 
\end{equation}   
where $\|\sigma_k\|=|\sigma_k|(\Bbb R^{n+m})$; 
\begin{equation} 
|\hat{\sigma}_k(\xi,\eta)| 
\leq C\|\Omega\|_q\|h\|_{\Delta_s}(\beta^{k+d} s(\xi))^{1/b_1},   
\end{equation}   
 where $d=b_1/\alpha_1$; 
\begin{equation}  
|\hat{\sigma}_k(\xi,\eta)| \leq C(\log \beta)\|\Omega\|_q\|h\|_{\Delta_s}
(\beta^{k} s(\xi))^{-\epsilon_0/(q's')}  
\end{equation}   
for some $\epsilon_0>0$; 
\begin{equation}  
 |\hat{\mu}_k(\xi,\eta)| \leq C(\log \beta)\|\Omega\|_q\|h\|_{\Delta_s}
(\beta^{k} s(\xi))^{-\epsilon_0/(q's')},    
\end{equation}   
where $\epsilon_0$ is as in (2.3); 
\begin{equation}   
 |\hat{\mu}_k(\xi,\eta) - \hat{\mu}_k(0,\eta)| \leq 
 C\|\Omega\|_q\|h\|_{\Delta_s}(\beta^{k+d} s(\xi))^{1/b_1},    
\end{equation}   
where $d$ is as in $(2.2)$. 
\par 
First we see that
\begin{equation}   
\|\sigma_k\|_1 = \int_{\beta^k}^{\beta^{k+1}}|h(r)|\|\Omega\|_1\, 
dr/r \leq C(\log \beta)\|\Omega\|_1\|h\|_{\Delta_1}. 
\end{equation}   
From this, (2.1) follows. 
Next, we show (2.2).  
Take $\nu\in \Bbb Z$ so that $2^\nu<\beta\leq 2^{\nu+1}$.  
Note that 
$$
\hat{\sigma}_k(\xi,\eta)=\int_{\beta^k<r(x)\leq \beta^{k+1}} 
 e^{-2\pi i \langle \Gamma(r(x)),\eta\rangle}
 (e^{-2\pi i \langle x,\xi\rangle}-1)h(r(x))\Omega(x')r(x)^{-\gamma}\,dx.   
 $$  
Thus 
\begin{align} 
|\hat{\sigma}_k(\xi,\eta)|  &\leq  C\int_{1<r(x)\leq \beta}   
 |x| |B_{\beta^k}\xi| |h(\beta^kr(x))\Omega(x')|
 r(x)^{-\gamma}\,dx                                    
\\ 
&\leq C\sum_{j=0}^\nu |B_{\beta^k}\xi| \|\Omega\|_1 
2^{j/\alpha_1}\int_{2^j}^{2^{j+1}}|h(\beta^kr)|\, dr/r    \notag 
\\ 
&\leq C\beta^{1/\alpha_1}
|B_{\beta^k}\xi| \|\Omega\|_1 \|h\|_{\Delta_1}.  \notag 
 \end{align}  
 Combining (2.6) and  (2.7), we have 
\begin{equation} 
 |\hat{\sigma}_k(\xi,\eta)|\leq C \|\Omega\|_1 \|h\|_{\Delta_1}
\min\left(\log\beta, \beta^{1/\alpha_1}|B_{\beta^k}\xi|\right).  
\end{equation}  
If $s(B_{\beta^k}\xi)< 1$, then $|B_{\beta^k}\xi|
\leq C(\beta^ks(\xi))^{1/b_1}$.   
Therefore,  
$$ \min\left(\log\beta, \beta^{1/\alpha_1}|B_{\beta^k}\xi|\right)\leq 
C(\beta^{k+d}s(\xi))^{1/b_1}.$$    
Using this in (2.8), we have (2.2).   
We can prove (2.5) in the same way.  
\par 
Next we prove (2.3). 
We use a method similar to that  of  \cite[p. 551]{DR}.    
 Define 
$$\tau(\xi)=\int_{\Sigma}\Omega(\theta)e^{-2\pi i\langle\xi, \theta\rangle}
\, d\sigma(\theta).   
$$  
 We need the following estimates.    
\begin{lemma} Let $L$ be the degree of the minimal polynomial of $P$.   
Then, if $0<\epsilon_0<a_2^{-1}\min(1/2,q'/L)$, we have 
$$\int_{\beta^{k}}^{\beta^{k+1}}\left|\tau(B_r\xi)\right|^2\, dr/r 
\leq C(\log \beta)(\beta^ks(\xi))^{-\epsilon_0/q'}\|\Omega\|_q^2,  $$ 
where $C$ is independent of  $\Omega\in L^q$, $q\in (1,2]$ and $\beta$. 
\end{lemma}   
In proving Lemma 1 we use the following estimate,  which follows from the 
corollary to Theorem 5 in Section 4 via an integration by parts argument.   
\begin{lemma} 
Let $L$ be as in Lemma $1$.   
Then, for $\eta, \zeta\in \Bbb R^n\setminus\{0\}$ we have 
$$ \left|\int_1^2 \exp\left(i\langle B_t\eta, \zeta\rangle\right)
\,dt/t \right| \leq C
\left|\langle \eta,P\zeta\rangle\right|^{-1/L}  $$ 
for some positive constant $C$ independent of $\eta$ and $\zeta$. 
\end{lemma}  
Proof of Lemma $1$.    
Choose  $\nu\in \Bbb Z$ such that $2^\nu<\beta\leq 2^{\nu+1}$.  
Then, we have 
\begin{align*} 
&\int_{\beta^{k}}^{\beta^{k+1}}\left|\tau(B_r\xi)\right|^2\, dr/r
\leq \sum_{j=0}^\nu \int_{\beta^{k}2^j}^{\beta^{k}2^{j+1}}  
\left|\tau(B_r\xi)\right|^2\, dr/r
\\ 
&= \sum_{j=0}^\nu 
\iint_{\Sigma\times \Sigma} \left(\int_{1}^{2}
\exp\left(-2\pi i\langle B_{\beta^k2^jr}\xi, \theta - \omega\rangle\right)
\, dr/r \right)\Omega(\theta)\bar{\Omega}(\omega)
\, d\sigma(\theta) \, d\sigma(\omega).   
\end{align*}  
By Lemma 2 we have 
$$ \left|\int_{1}^{2}
\exp\left(-2\pi i\langle B_{\beta^k2^jr}\xi, \theta - \omega\rangle\right)
\, dr/r\right|  \leq C\left|\left\langle B_{\beta^k2^j}\xi,P(\theta-\omega)
\right\rangle\right|^{-\epsilon}, 
$$ 
where  $0<\epsilon \leq 1/L$.   
Using H\"{o}lder's inequality, if $0<\epsilon<\min(1/(2q'),1/L)$, we see  that 
\begin{multline*} 
\iint_{\Sigma\times \Sigma}\left|\langle B_{\beta^k2^j}\xi,
P(\theta - \omega)\rangle\right|^{-\epsilon}
\left|\Omega(\theta)\bar{\Omega}(\omega)\right|
\, d\sigma(\theta) \, d\sigma(\omega) 
\\ 
\leq \left(\iint_{\Sigma\times \Sigma}\left|
\langle P^*B_{\beta^k2^j}\xi,
\theta - \omega\rangle\right|^{-\epsilon q'}
\, d\sigma(\theta) \, d\sigma(\omega)\right)^{1/q'}\|\Omega\|_q^2
\leq C|B_{\beta^k2^j}\xi|^{-\epsilon}\|\Omega\|_q^2, 
\end{multline*}    
where the last inequality follows from  (3) of Section 1 
(see \cite[p. 553]{DR}).  
Therefore 
\begin{equation} 
\int_{\beta^{k}}^{\beta^{k+1}}\left|\tau(B_r\xi)\right|^2\, dr/r 
\leq C\|\Omega\|_q^2\sum_{j=0}^\nu|B_{\beta^k2^j}\xi|^{-\epsilon} 
\quad (0<\epsilon<\min(1/(2q'),1/L)). 
\end{equation} 
If $s(B_{\beta^k}\xi)\geq 1$, 
 $|B_{\beta^k2^j}\xi|\geq C(\beta^k2^js(\xi))^{1/a_2}$ ($0\leq j\leq \nu$). 
 Thus  we see that 
\begin{equation} 
 \sum_{j=0}^\nu|B_{\beta^k2^j}\xi|^{-\epsilon} 
 \leq \sum_{j=0}^\nu C(\beta^k2^js(\xi))^{-\epsilon/a_2}
 \leq  C(\log\beta)(\beta^k s(\xi))^{-\epsilon/a_2},     
 \end{equation} 
 where $C$ is independent of $q$.   
By (2.9) and (2.10) we have the estimate of Lemma 1 when 
$s(B_{\beta^k}\xi)\geq 1$. If   $s(B_{\beta^k}\xi) < 1$, 
the estimate of Lemma 1 follows from the inequality 
$|\tau(\xi)|\leq \|\Omega\|_1$.  This completes the proof of Lemma 1.  
\par 
Now, by H\"older's inequality we have 
\begin{align} 
|\hat{\sigma}_k(\xi,\eta)| 
&= \left| \int_{\beta^{k}}^{\beta^{k+1}} 
e^{-2\pi i \langle \Gamma(r),\eta\rangle}
h(r) \tau(B_r\xi)\, 
dr/r \right|            \\
&\leq \left(\int_{\beta^{k}}^{\beta^{k+1}}|h(r)|^s\,dr/r \right)^{1/s}
\left(\int_{\beta^{k}}^{\beta^{k+1}}\left|\tau(B_r\xi)\right|^{s'}\,
dr/r\right)^{1/s'}      \notag 
\\ 
&\leq C(\log \beta)^{1/s}\|h\|_{\Delta_s}
\|\Omega\|_1^{(s'-2)/s'}\left(\int_{\beta^{k}}^{\beta^{k+1}}
\left|\tau(B_r\xi)\right|^2\, dr/r\right)^{1/s'},  
\notag 
\end{align}  
where we have used the estimate $|\tau(\xi)|\leq \|\Omega\|_1$ to get 
the last inequality.   
By (2.11) and Lemma 1 we have (2.3).   The estimate (2.4) can be proved 
similarly.  
\par 
Let $B_{qs}=(1-\beta^{-\theta \epsilon_0/(q's')})^{-1}$, where 
$\beta\geq 2$,  $\theta \in (0,1)$ and $\epsilon_0$ is as in (2.3) and 
(2.4).    
To prove Theorems 1 and 3, we use the following:
\begin{proposition}  Suppose that $\Omega\in L^q$, $q\in (1,2]$ and 
$h\in \Delta_s$, $s\in (1,2]$.   Let $|1/p-1/2|<(1-\theta)/(s'(1+\theta))$.  
Then,  we have 
$$\|Tf\|_p \leq C(\log\beta)\|h\|_{\Delta_s}\|\Omega\|_q 
B_{qs}B_{q2}^{|1/p-1/p'|}\|f\|_p,  $$ 
where $C$ is a constant independent of $\Omega$, $h$, $q$, $s$ and $\beta$. 
\end{proposition}   
\begin{proposition} Suppose  that $\Gamma\equiv 0$. 
 Let $\Omega\in L^q$, $h\in \Delta_s$, 
$q, s \in (1,2]$. Then, for  $p \in (1+\theta, (1+\theta)/\theta)$ we have 
$$\|Tf\|_p \leq C(\log \beta)\|\Omega\|_q\|h\|_{\Delta_s}
B_{qs}^{1 + |1/p - 1/p'|}\|f\|_p, $$ 
where $C$ is a constant independent of $\Omega$, $h$, $q$, $s$ 
and $\beta$. 
\end{proposition} 
To prove 
Propositions 1 and 2, we need the following:  
\begin{proposition} Let $\mu^*(f)(x) = \sup_k |\mu_k*f(x)|$.  
Let $\Omega\in L^q$, $q\in (1,2]$.  
\begin{enumerate}  
\renewcommand{\labelenumi}{(\arabic{enumi})}  
\item 
If $h\in \Delta_\infty$, for $p >1+\theta$  we have 
$$\|\mu^*(f)\|_p \leq C(\log \beta)\|\Omega\|_q\|h\|_{\Delta_{\infty}}
B_{q2}^{2/p}\|f\|_p,   $$ 
where  $C$ is a constant independent of $\Omega$, $h$,  $q$ and $\beta$. 
 \item 
Suppose that $\Gamma\equiv 0$. Let $h\in \Delta_s$, $s\in(1,2]$. Then,
 we have 
 $$\|\mu^*(f)\|_p \leq C(\log \beta)\|\Omega\|_q\|h\|_{\Delta_{s}}
B_{qs}^{2/p}\|f\|_p   $$   
for $p >1+\theta$, where $C$ is independent of $\Omega$, $q$, $h$, $s$ 
 and $\beta$.  
\end{enumerate}  
\end{proposition} 
\begin{proof} Since the estimate 
$\|\mu^*(f)\|_\infty\leq C(\log\beta)\|\Omega\|_1\|h\|_{\Delta_1}
\|f\|_\infty$ follows from (2.1), by interpolation, to prove (1) and (2) of 
Proposition 3 we may assume $p\in(1+\theta,2]$.  
\par 
First, we give a proof of part (1).  
Define measures $\nu_k$ on $\Bbb R^n\times\Bbb R^m$ by 
$$\hat{\nu}_k(\xi,\eta) = \hat{\mu}_k(\xi,\eta) - \hat{\Psi}_k(\xi,\eta),$$  
where $\hat{\Psi}_k(\xi,\eta)=\hat{\varphi}_k(\xi)\hat{\mu}_k(0,\eta)$ with  
 $\varphi_k (x) = \beta^{-k\gamma}\varphi (A_{\beta^{-k}}x), 
\varphi \in C_0^{\infty}$.  
We assume that $\varphi$ is supported in $\{r(x) \leq 1 \},     
\hat{\varphi}(0) = 1$ and $ \varphi \geq 0$.  
Then  by (2.1), (2.4) and (2.5), for $q, s\in (1,2]$,  we have 
$$\left|\hat{\nu}_k(\xi,\eta) \right| \leq 
C(\log \beta)\|\Omega\|_q\|h\|_{\Delta_s} 
\min \left(1, (\beta^{k+d}s(\xi))^{1/b_1}, 
(\beta^{k} s(\xi))^{-\epsilon_0/(q's')} \right).   $$ 
We may assume that $\epsilon_0$ is small enough so that 
$\epsilon_0/4\leq 1/b_1$.   Then, we see that  
\begin{equation} 
\left|\hat{\nu}_k(\xi,\eta) \right| \leq 
CA\min \left(1, (\beta^{k+d}s(\xi))^{\alpha}, 
(\beta^{k} s(\xi))^{-\alpha} \right),  
\end{equation}
where $A=(\log \beta)\|\Omega\|_q\|h\|_{\Delta_\infty}$ and 
$\alpha=\epsilon_0/(2q')$.    
\par 
Let
$$g(f)(x,z) = \left( \sum_{k=-\infty}^\infty \left|\nu_k *f(x,z) \right |^2 
\right)^{1/2}.$$
Then $\mu^*(f) \leq g(f) + \Psi^*(|f|)$, 
where $\Psi^*(f)=\sup_k||\Psi_k|*f|$.  
Let 
   $$Mg(x)= \sup_{t>0}t^{-\gamma}\int_{r(x-y)<t}|g(y)|\,dy$$ 
 be the Hardy--Littlewood maximal function on $\Bbb R^n$ with respect to 
 the function $r$. 
By the $L^p$ boundedness of $M_\Gamma$ and $M$, 
it is easy to see that 
$\|\Psi^*(f)\|_p\leq CA\|f\|_p$ for $p>1$.  Thus to prove 
Proposition 3 (1) it suffices to show  
\begin{equation} 
\|g(f)\|_p \leq CAB^{2/p}\|f\|_p  
\quad (p\in (1+\theta,2]),  
\end{equation}   
where $A$ is as above and $B=B_{q2}$.  
By a well-known property of Rademacher's functions,  (2.13) follows from  
\begin{equation} 
 \left \|U_\epsilon(f) \right \|_p \leq CA B^{2/p}
\|f\|_p \quad (p\in (1+\theta,2]), 
\end{equation} 
where $ U_\epsilon(f)(x,z) = \sum \epsilon_k \nu_k * f(x,z)$  
with $\epsilon = \{\epsilon_k \}$, $\epsilon_k = 1$ or $-1$  
(the inequality is uniform in $\epsilon$). 
\par
We define two sequences $\{r_m\}_1^{\infty}$ 
and $\{p_m \}_1^{\infty}$ by $p_1 = 2$ and
$$\frac{1}{r_m} - \frac 12 = \frac{1}{2p_m}, \quad \frac{1}{p_{m+1}} 
= \frac{\theta}{2} + \frac{1 - \theta}{r_m} \qquad \text{for $m \geq 1$.}$$
Then,  we have
$$\frac{1}{p_{m+1}} = \frac{1}{2} + \frac{1 - \theta}{2p_m}
 \qquad \text{for $m \geq 1$}. $$
 Thus   $1/p_m=(1-\eta^m)/(1+\theta)$,  where $\eta=(1-\theta)/2$, so 
$\{p_m\}$ is decreasing and converges to $1 + \theta$. 
\par 
   For $j \geq 1$ we prove   
\begin{equation} 
\left \|U_{\epsilon}(f)\right \|_{p_j} 
\leq C_jAB^{2/p_j}\left \|f\right \|_{p_j}.   
\end{equation} 
 To prove (2.15) we use the Littlewood--Paley theory. 
Let $\{\psi_k\}_{- \infty}^{\infty}$ be a sequence 
of non-negative functions in
 $C^{\infty}((0, \infty))$  such that 
 \begin{gather*}
 \text{supp} (\psi_k) \subset [\beta^{-k-1}, \beta^{-k+1}], \quad
\sum_k \psi_k(t)^2 = 1, 
\\
|(d/dt)^j\psi_k (t)| \leq c_j/t^j \quad (j=1,2, \dots), 
\end{gather*}    
where $c_j$ is independent of $\beta\geq 2$.   
Define $S_k$ by
$$\left(S_k(f)\right)\hat{\phantom{f}}(\xi,\eta) = 
 \psi_k(s(\xi))\hat{f}(\xi,\eta). $$
We write $U_\epsilon(f)=\sum_{j=-\infty}^\infty U_j(f)$, where 
$U_j(f) =  \sum_{k=-\infty}^{\infty}\epsilon_k S_{j+k}\left(\nu_k 
* S_{j+k}(f)\right)$.  
Then by Plancherel's theorem and (2.12) we have
\begin{align}
\left\|U_j(f) \right\|_2^2 
&\leq \sum_k C\iint_{D(j + k)\times \Bbb R^m} 
|\hat{\nu}_k(\xi,\eta)|^2|\hat{f}(\xi,\eta)|^2 
\, d\xi\,d\eta 
\\
&\leq CA^2\min\left(1,\beta^{-2(|j|-1-d)\alpha}\right) \sum_k  
\iint_{D(j + k)\times \Bbb R^m} |\hat{f}(\xi,\eta)|^2 \, d\xi\,d\eta \notag 
\\
&\leq CA^2\min\left(1,\beta^{-2(|j|-1-d)\alpha}\right)\|f\|^2_2,  \notag 
\end{align}  
where $D(k) = \{\xi\in \Bbb R^n: \beta^{-k-1} < s(\xi) 
\leq \beta^{-k+1}\}$.  
By (2.16) we have 
\begin{align}  
\left\|U_{\epsilon}(f) \right\|_2 &\leq \sum_{-\infty}^{\infty}\|U_j(f)\|_2 
\leq C\sum_{-\infty}^{\infty} 
A\min\left(1,\beta^{-(|j|-1-d)\alpha}\right)\|f\|_2    
\\
& \leq CA(1 - \beta^{- \alpha})^{-1}\|f\|_2.    \notag 
\end{align}  
If we denote by $A(m)$ the estimate of (2.15) for $j=m$, this  proves $A(1)$.  
\par 
Now,  we assume $A(m)$ and derive $A(m+1)$ from $A(m)$.  Note that
$$ \nu^*(f) \leq \mu^*(|f|) + \Psi^*(|f|) 
\leq g(|f|)(x) + 2\Psi^*(|f|), $$
where $\nu^*(f)(x) = \sup_k ||\nu_k|*f(x)|$.   Since 
$\|g(f)\|_{p_m} \leq CAB^{2/p_m}\|f\|_{p_m}$ by $A(m)$, we have
$$\|\nu^*(f)\|_{p_m} \leq CAB^{2/p_m}\|f\|_{p_m}.$$  
Also,  $\|\nu_k\|\leq CA$ by (2.1).  
Thus, by the proof of Lemma for Theorem B in \cite[p. 544]{DR}, we have 
the vector valued inequality:
\begin{align}  
\left\|\left(\sum |\nu_k*g_k|^2\right)^{1/2}\right\|_{r_m} 
&\leq C(AB^{2/p_m}\sup_k\|\nu_k\|)^{1/2}
\left\|\left(\sum |g_k|^2\right)^{1/2}\right\|_{r_m}   
\\ 
&\leq CAB^{1/p_m}\left\|\left(\sum |g_k|^2\right)^{1/2}\right\|_{r_m}. \notag 
\end{align} 
By (2.18) and the Littlewood--Paley inequality, we have
\begin{align} 
\|U_j(f)\|_{r_m} 
&\leq C\left\| \left(\sum_k |\nu_k*S_{j+k}(f)|^2\right)^{1/2}\right\|_{r_m} 
\\
 &\leq CAB^{1/p_m}\|f\|_{r_m}.   \notag 
\end{align}    
Here we note that the bounds for the Littlewood-Paley inequality are 
independent of $\beta\geq 2$. 
Interpolating between (2.16) and (2.19), we have
$$\|U_j(f)\|_{p_{m+1}} 
\leq CAB^{(1 - \theta)/p_m}\min\left(1,\beta^{-\theta \alpha(|j|-1-d)}\right)
\|f\|_{p_{m+1}}.$$
Thus
\begin{align*}  
\|U_{\epsilon}(f)\|_{p_{m+1}} &\leq \sum_j \|U_j(f)\|_{p_{m+1}}    
\leq CAB^{(1 - \theta)/p_m}(1-\beta^{-\theta \alpha})^{-1}\|f\|_{p_{m+1}}
\\ 
&\leq CAB^{2/p_{m+1}}\|f\|_{p_{m+1}}, 
\end{align*}        
which proves $A(m + 1)$.  By induction, this completes the proof of (2.15). 
\par  
Now we prove (2.14). Let $p \in (1+\theta, 2]$  and let $\{p_m\}_1^{\infty}$ 
be as in (2.15).  Then we have $p_{N+1} < p \leq p_N$ 
for some $N$. 
By interpolation between the estimates in (2.15) for $j=N$ and $j=N+1$ 
we have (2.14).   This completes the proof of  Proposition 3 (1). 
\par 
Part (2) of Proposition 3 can be proved in the same way. 
 We  take $A=(\log \beta)\|\Omega\|_q\|h\|_{\Delta_s}$ and  
 $\alpha=\epsilon_0/(q's')$ in (2.12).   
 Then, since 
$$\|\Psi^*(f)\|_p\leq C(\log \beta)\|\Omega\|_1\|h\|_{\Delta_1}\|f\|_p 
\qquad \text{for  $p>1$}$$ 
 if $\Gamma\equiv 0$,  the proof of part (1) can be used to get  
 (2.13) with $A=(\log \beta)\|\Omega\|_q\|h\|_{\Delta_s}$ as above 
 and  $B=B_{qs}$, and the conclusion of part (2) follows from (2.13). 
\end{proof} 
\par 
Proof of Proposition $1$.    
To prove Proposition 1 we may assume $1<s<2$. As in \cite{A-SP}, 
here we apply an idea in  the proof of  \cite[Theorem 7.5]{FP}.  
We consider measures $\tau_k$  defined by 
$$\hat{\tau}_k(\xi,\eta)=\int_{E_k}e^{-2\pi i\langle y,\xi\rangle} 
e^{-2\pi i\langle\Gamma(r(y)),\eta\rangle} 
|h(r(y))|^{2-s} |\Omega(y')|r(y)^{-\gamma}\, dy.$$  
Then,  the Schwarz inequality implies 
\begin{equation} 
|\sigma_k*f|^2\leq C(\log\beta)\|h\|_{\Delta_s}^s\|\Omega\|_1\tau_k*|f|^2. 
\end{equation}  
Define measures $\lambda_k$ by 
$$\hat{\lambda}_k(\xi,\eta)=\int_{E_k}e^{-2\pi i\langle y,\xi\rangle} 
e^{-2\pi i\langle\Gamma(r(y)),\eta\rangle} |\Omega(y')|r(y)^{-\gamma}\, dy.$$  
Since $|h|^{2-s}\in \Delta_{s/(2-s)}$ and 
$\||h|^{2-s}\|_{\Delta_{s/(2-s)}}=\|h\|_{\Delta_s}^{2-s}$, if  
$u=s/(2-s)$ by H\"{o}lder's inequality we have  
$$
|\tau_k*f|\leq C(\log \beta)^{1/u}\|h\|_{\Delta_s}^{2-s}\|\Omega\|_1^{1/u}
(\lambda_k*|f|^{u'})^{1/u'}.   
 $$  
Therefore, if $1+\theta<r/u'=2r(s-1)/s$,  by applying 
 (1) of Proposition 3 to $\{\lambda_k\}$ we see that 
\begin{equation} 
\|\tau^*(f)\|_r\leq C(\log\beta)\|h\|_{\Delta_s}^{2-s}\|\Omega\|_q 
B_{q2}^{2/r}\|f\|_r,  
\end{equation}  
where $\tau^*(f)=\sup_k|\tau_k*f|$. 
Thus, if $|1/v-1/2|=1/(2r)<1/(s'(1+\theta))$, 
using (2.20), (2.21) and arguing as in the proof of 
Lemma for Theorem B in \cite[p. 544] {DR}, we see that 
\begin{equation} 
\left\|\left(\sum |\sigma_k*g_k|^2\right)^{1/2}\right\|_{v} 
\leq C(\log\beta)\|h\|_{\Delta_s}\|\Omega\|_q B_{q2}^{1/r}
\left\|\left(\sum |g_k|^2\right)^{1/2}\right\|_{v}. 
\end{equation}  
\par 
We decompose 
$Tf=\sum_{j=-\infty}^{\infty}V_jf$, where $V_jf
=\sum_{k=-\infty}^{\infty}S_{j+k}\left(\sigma_k * S_{j+k}(f)\right)$.   
Then,  using (2.22) and the Littlewood--Paley theory, we see that  
\begin{equation} 
\|V_jf\|_v\leq C(\log\beta)\|h\|_{\Delta_s}\|\Omega\|_q
 B_{q2}^{1/r}\|f\|_v,  
 \end{equation} 
where $|1/v-1/2|=1/(2r)<1/(s'(1+\theta))$. On the other hand, 
by (2.1)--(2.3) we have 
$$\left|\hat{\sigma}_k(\xi,\eta) \right| \leq 
C(\log \beta)\|\Omega\|_q\|h\|_{\Delta_s} 
\min \left(1, (\beta^{k+d}s(\xi))^{\kappa}, 
(\beta^{k} s(\xi))^{-\kappa} \right),    $$ 
where $\kappa=\epsilon_0/(q's')$, and hence,  similarly to the proof of   
(2.16), we can show that  
\begin{equation} 
\|V_jf\|_2\leq C(\log\beta)\|h\|_{\Delta_s}\|\Omega\|_q 
\min\left(1,\beta^{-(|j|-1-d)\kappa}\right)\|f\|_2.    
\end{equation}   
If $|1/p-1/2|<(1-\theta)/(s'(1+\theta))$, then we can find 
 numbers $v$ and $r$ such that 
$|1/v-1/2|=1/(2r)<1/(s'(1+\theta))$ and  $1/p=\theta/2+(1-\theta)/v$. 
Thus, interpolating between (2.23) and (2.24), 
we have 
$$\|V_jf\|_p\leq C(\log\beta)\|h\|_{\Delta_s}\|\Omega\|_q B_{q2}^{(1-\theta)/r}
\min\left(1,\beta^{-\theta(|j|-1-d)\kappa}\right)\|f\|_p. $$  
Therefore  
\begin{equation} 
\|Tf\|_p\leq \sum_j\|V_jf\|_p\leq 
C(\log\beta)\|h\|_{\Delta_s}\|\Omega\|_q B_{q2}^{(1-\theta)/r}B_{qs}\|f\|_p. 
\end{equation}   
This completes the proof of Proposition 1, since  $(1-\theta)/r=|1/p-1/p'|$.  
\par 
Proof of Proposition $2$. 
The $L^2$ estimates follow from Proposition 1, so 
 on account of duality and 
interpolation  we may assume that $1+\theta < p \leq 4/(3-\theta)$. 
For  $p_0 \in (1+\theta, 4/(3-\theta)]$ we can find $r\in (1+\theta,2]$ 
such that $1/p_0=1/2+(1-\theta)/(2r)$.  
If $\Gamma\equiv 0$,  by (2) of Proposition 3 and (2.1), arguing as in (2.18), 
 we have (2.22) 
with $B_{q2}$ replaced by $B_{qs}$  for the number  
$v$ satisfying  $1/v-1/2=1/(2r)$ (note that $1/p_0=\theta/2+(1-\theta)/v$).  
Thus, arguing as in the proof of 
Proposition 1, we have (2.25) with $p=p_0$ and $B_{qs}$ in place of $B_{q2}$.  
This completes the proof of Proposition 2. 
\par  
Now we can give  proofs of Theorems $1$ and $3$.   
To prove Theorem 1, we may assume that $1<s\leq 2$. 
Let  $\beta=2^{q'}$ in Proposition 1.  
Then, since $\theta$  is  an arbitrary number in $(0,1)$,  
we have Theorem 1 for $s\in (1,2]$. 
\par 
Next, take  
$\beta=2^{q's'}$ in Proposition 2.  Then, we have 
$$\|Tf\|_p \leq C(q - 1)^{-1}(s - 1)^{-1}\|\Omega\|_q\|h\|_{\Delta_s}\|f\|_p $$ for $p \in (1, \infty)$, since $(1+\theta, (1+\theta)/\theta)\to (1,\infty)$ 
as $\theta\to 0$.   
From this the result for  $S$ in Theorem 3 
follows if we take functions of the form 
$f(x,z)=k(x)g(z)$.

\section{Extrapolation}  

We can prove Theorems 2 and 4 by an extrapolation method  similar to the 
one used in \cite{S}.  
We give a proof of Theorem 4 for the sake of completeness (Theorem 2 can be 
proved in the same way).  
We fix $p\in(1,\infty)$ and $f$ with $\|f\|_p\leq 1$.  
Let $S$ be as in (1.2). 
We also write $Sf=S_{h,\Omega}(f)$.  
Put $U(h,\Omega) = \|S_{h,\Omega}(f)\|_p$.  
Then we see that 
\begin{equation}\begin{split} 
U(h, \Omega_1+\Omega_2) &\leq U(h, \Omega_1) + U(h, \Omega_1),  
 \\ 
 U(h_1 + h_2, \Omega) &\leq U(h_1, \Omega) + U(h_2, \Omega),  
 \end{split}\end{equation} 
 for appropriate functions $\Omega, h, \Omega_1, \Omega_2, h_1$ and $h_2$. 
Set 
\begin {align*}
E_1 &= \{r \in \Bbb R_+ : |h(r)| \leq 2 \}, 
\\ 
E_m &= \{r \in \Bbb R_+ : 2^{m-1} < |h(r)| \leq 2^{m}\} \qquad   
\text{for $m \geq 2$}. 
\end{align*}  
  Then $h=\sum_{m=1}^\infty h\chi_{E_m}$.  
Put $e_m=\sigma(F_m)$ for $m\geq 1$, where 
\begin{align*}
F_m&=\{\theta\in \Sigma: 
2^{m-1}<|\Omega(\theta)|\leq 2^m\} \qquad \text{for $m\geq 2$}, 
\\ 
F_1&=\{\theta\in \Sigma: |\Omega(\theta)|\leq 2\}. 
\end{align*}  
Let $\Omega_m=\Omega\chi_{F_m}-\sigma(\Sigma)^{-1}\int_{F_m}\Omega\,d\sigma$. 
Then $\Omega=\sum_{m=1}^\infty\Omega_m$. 
Note that $\int_\Sigma \Omega_m\,d\sigma =0$. 
Applying Theorem 3, we see that  
\begin{equation} 
U\left(h\chi_{E_m},\Omega_j\right) 
\leq C(q - 1)^{-1}(s - 1)^{-1}\|h\chi_{E_m}\|_{\Delta_s}\|\Omega_j\|_{q} 
\end{equation} 
for all $s, q \in (1, 2]$. 
\par 
Now we follow the extrapolation argument of A.~Zygmund 
\cite[Chap. XII, pp. 119--120]{Z}.  For $k\in \Bbb Z$,  put 
\begin{align*}
E(k, m) &= \{r \in (2^{k}, 2^{k+1}] : 2^{m-1} < |h(r)| \leq  2^{m}\} 
\qquad \text{for $m\geq 2$},  
\\ 
E(k, 1) &= \{r \in (2^{k}, 2^{k+1}] : 0<|h(r)| \leq  2 \}. 
\end{align*}  
  Then  
\begin{align*} 
\int_{E(k,m)}|h(r)|^{(m+1)/m} dr/r &\leq Cm^{-a}\int_{E(k,m)}
|h(r)|\left(\log(2+|h(r)|)\right)^{a}\, dr/r 
\\ 
&\leq Cm^{-a}L_{a}(h),
\end{align*}  
and hence  
\begin{equation} 
\|h\chi_{E_m}\|_{\Delta_{1+1/m}} \leq Cm^{-am/(m+1)}L_{a}(h)^{m/(m+1)} 
\end{equation} 
for $ m \geq 1$.  Also we have 
\begin{equation} 
\|\Omega_j\|_{1+1/j}\leq C2^j e_j^{j/(j+1)}. 
\end{equation}  
From  (3.1)--(3.4) we see that
\begin{align*}  
U(h,\Omega)&\leq \sum_{m \geq 1}\sum_{j \geq 1} U\left(h\chi_{E_m}, 
\Omega_j\right) 
\leq C\sum_{m \geq 1}\sum_{j \geq 1} 
 jm\|h\chi_{E_m}\|_{\Delta_{1+1/m}}\|\Omega_j\|_{1+1/j} \\
&\leq C(1+L_a(h))\sum_{m \geq 1} \sum_{j \geq 1}
m^{1-am/(m+1)}j2^j e_j^{j/(j+1)}
\\
&= C(1+L_a(h))\left(\sum_{m \geq 1} m^{1-am/(m+1)}\right)\left(\sum_{j \geq 1}
j2^j e_j^{j/(j+1)}\right).  
\end{align*}   
When $a>2$, it is easy to see that $\sum_{m \geq 1} m^{1-am/(m+1)}<\infty$. 
Also, we have 
\begin{align*} 
 \sum_{j\geq 1}j 2^j e_j^{j/(j+1)} 
&= \sum_{e_j < 3^{-j}} + \sum_{e_j \geq 3^{-j}}    
\\
&\leq \sum_{j \geq 1}j2^j3^{-j^2/(j+1)}
 + \sum_{j \geq 1}j2^j e_j3^{j/(j+1)}
 \\
&\leq C+C\int_{\Sigma}|\Omega(\theta)|
\log(2+|\Omega(\theta)|)\, d\sigma(\theta).
\end{align*}     
Collecting the results, we conclude  the proof of Theorem 4.     

\begin{remarkz} 
For a positive number $a$ and a function $h$ on $\Bbb R_+$,  let  
$$
 N_a(h)  = \sum_{m \geq 1}m^a 2^md_m(h), $$
where 
 $d_m(h) = \sup_{k \in \Bbb Z} 2^{-k}|E(k,m)|$ ($E(k,m)$ is as above). 
We define a class $\mathscr N_a$ to be the space of all measurable functions 
$h$ on $\Bbb R_+$ which satisfy  $N_a(h)<\infty$.  
Then, it can be shown that if $h\in \mathscr L_{a}$ for some $a > 2$, 
then $h\in \mathscr N_1$.  By a method similar to that used in this section, 
we can show the $L^p$ boundedness of $S$ in Theorem 4 under a less restrictive 
condition that $h\in \mathscr N_1$ and $\Omega\in L\log L$ (see \cite{S}).  
\end{remarkz}   

\section{An estimate for  a trigonometric integral} 

Let $A$ be an $n\times n$ real matrix and 
$$\phi_A(t)=(t-\gamma_1)^{m_1}(t-\gamma_2)^{m_2}\dots (t-\gamma_k)^{m_k} $$
be the minimal polynomial of $A$, where $\gamma_i\neq\gamma_j$ if $i\neq j$.  
 Let $a_i(t)=(t-\gamma_i)^{m_i}$ for $i=1,2,\dots ,k$.  
Then, we can find polynomials $b_i(t)$ $(i=1,2,\dots,k)$ such that 
$$\frac{1}{\phi_A(t)}=\sum_{i=1}^k\frac{b_i(t)}{a_i(t)}. $$ 
For each $i$, $1\leq i\leq k$, let  $P_i$  be the polynomial defined by 
$$P_i(t)=\frac{b_i(t)}{a_i(t)}\phi_A(t).   $$ 
 We consider the $n\times n$ matrices $P_i(A)$, which are defined as usual 
(see \cite{FIS}).   
\par  
 Let 
$$
V_i= \{z\in \Bbb C^n : (A-\gamma_i E)^{m_i}z=0\} \quad 
(i=1,2,\dots , k),  $$  
where  $E$ denotes the unit matrix.   
Then, the vector space $\Bbb C^n$ can be decomposed into a direct sum  as   
$$\Bbb C^n=V_1\oplus V_2\oplus\dots \oplus  V_{k}. $$  
Each of the matrices $P_i(A)$ is the projection onto $V_i$; indeed, 
we have the following (see \cite{FIS}): 
$P_i(A)z\in V_i$ for all $z\in \Bbb C^n$, for $i=1,2,\dots, k$,  and 
\begin{gather*} 
P_1(A)+P_2(A)+\dots +P_{k}(A)=E, 
\\ 
P_i^2(A)=P_i(A), \quad P_i(A)P_j(A)=0 \quad\text{if $i\neq j$}\quad 
(1\leq i, j\leq k). 
\end{gather*} 
\par 
 For $z=(z_i)$ and $w=(w_i)$ in $\Bbb C^n$, we write $\langle z,w\rangle= 
\sum_{i=1}^nz_iw_i$.  
Let 
\begin{equation}
J(A,\eta,\zeta)=
\sum_{i=1}^{k}\sum_{j=0}^{m_i-1}\left|\langle (A-\gamma_i E)^jP_i(A)\eta, 
A^*\zeta\rangle\right|   \label{c1.1}
\end{equation} 
for $\eta, \zeta\in \Bbb R^n$.  
 In this section, we prove the following: 
 \begin{theorem}     
  Let  $\eta, \zeta\in \Bbb R^n\setminus\{0\}$ and $0<a<b$.  
  Suppose that $J(A,\eta,\zeta)\neq 0$ and the numbers $a$, $b$  are in a 
  fixed compact  subinterval of $(0,\infty)$. 
Then, we have     
$$ \left|\int_a^b \exp\left(i\langle t^A\eta,\zeta\rangle\right)
\, dt\right| 
\leq CJ(A,\eta,\zeta)^{-1/N},   $$ 
where  $N=\deg \phi_A =m_1+m_2+\dots+m_k$ and 
the constant $C$ is independent of $\eta$, $\zeta$,  $a$ and $b$. 
\end{theorem}   
Since $\sum_{i=1}^kP_i(A)=E$, using the triangle inequality, we see that 
$$\left|\langle \eta, A^*\zeta\rangle\right|\leq  \sum_{i=1}^{k}
\left|\langle P_i(A)\eta, A^*\zeta\rangle\right|\leq J(A,\eta,\zeta). $$   
Therefore,  Theorem 5 implies the following:   
  \begin{corollaryz} 
  Let $\eta, \zeta, a, b$ and $N$ be as in Theorem $5$. Then, we have 
  $$ \left|\int_a^b \exp\left(i\langle t^A\eta,\zeta\rangle\right)
\, dt\right| 
\leq C\left|\langle A\eta, \zeta\rangle\right|^{-1/N}$$ 
when $\langle A\eta, \zeta\rangle\neq 0$.  
\end{corollaryz} 
This is used to prove Lemma 2 in Section 2.   
\par 
We define the curve $X(t)=t^A\eta$ for a fixed 
$\eta\in\Bbb R^n\setminus\{0\}$. 
Then,  E.~M.~Stein and S.~Wainger \cite{SW2} proved the following (see 
\cite{NW, SW} for related results): 
\begin{theorema}
Suppose that the curve $X$  does not lie in an affine hyperplane.  
     Then    
$$ \left|\int_a^b \exp\left(i\langle X(t),\zeta\rangle\right)
\, dt\right| 
\leq C\left|\zeta\right|^{-1/n},   $$ 
where $C$ is independent of $\zeta\in \Bbb R^n\setminus\{0\};$ 
 furthermore, 
if  $a$ and $b$  are in a fixed compact subinterval of $(0,\infty)$, 
the constant $C$ is  also independent of $a$ and $b$.  
\end{theorema} 
\par 
 Now, we see that Theorem 5 implies Theorem A.    
Since $P_i(A)z\in V_i$  ($z\in \Bbb C^n$), we have 
$(A-\gamma_i E)^m P_i(A)=0$ if $m\geq m_i$ ($i=1,2,\dots , k$). 
 Therefore 
 \begin{align*}     
 \exp((\log t)A)P_i(A)&=\exp((\log t)\gamma_iE)
 \exp((\log t)(A-\gamma_iE))P_i(A)                     
 \\
 &=t^{\gamma_i}\sum_{j=0}^{m_i-1}\frac{(\log t)^j}{j!}(A-\gamma_i E)^jP_i(A).  
 \end{align*}   
Thus, using $\sum_{i=1}^{k}P_i(A)=E$, we see that 
 \begin{equation} 
 t^A=\sum_{i=1}^{k}t^{\gamma_i}
 \left[\sum_{j=0}^{m_i-1}\frac{(\log t)^j}{j!}(A-\gamma_i E)^j\right]P_i(A).  
 \label{c1.2}
\end{equation}  
\par  
 The assumption on $X$ of Theorem A can be rephrased as follows:  
 the function $\psi(t)=\langle t^A\eta,\zeta\rangle$ is not a constant function on $(0,\infty)$ for every $\zeta\in \Bbb R^n\setminus\{0\}$. 
 If  $\psi(t)$ is not a constant function, 
then $\psi'(t)$ is not identically $0$. 
Thus, since  $t(d/dt)\psi(t)=\langle t^A\eta,A^*\zeta\rangle$, by \eqref{c1.2}
 we have $J(A,\eta,\zeta)>0$, where $J(A,\eta,\zeta)$ is as in \eqref{c1.1}. 
 Let $C_0=\min_{|\zeta|=1}J(A,\eta,\zeta)$ and note that 
 $C_0>0$.  Then, from  Theorem 5, it 
 follows that 
 $$ \left|\int_a^b \exp\left(i\langle X(t),\zeta\rangle\right)
\, dt\right| 
\leq CC_0^{-1/N}|\zeta|^{-1/N}.   $$ 
This implies  Theorem A, since $N\leq n$ 
(in fact, it is not difficult to see that $N=n$ if $X$ satisfies the 
assumption of Theorem A).  
\par 
In the following,  we  give a proof of  Theorem 5.   
Let $I=[\alpha,\beta]$ be a compact interval in $\Bbb R$.  Consider the 
differential equation 
\begin{equation} 
y^{(k)}+a_1y^{(k-1)}+a_2y^{(k-2)}+\dots +a_ky=0 \quad \text{on $I$},   
\label{c2.1}  
\end{equation}  
where $a_1, a_2, \dots , a_k$ are complex constants.  
 Let $\{\varphi_1, \varphi_2, \dots , \varphi_k\}$ be a basis for the space 
$S$ of all solutions of \eqref{c2.1}.  Then, we use the following  to prove 
 Theorem 5.   
\begin{proposition} 
  Let $\varphi$ be a real valued 
function such that  $\varphi'\in S$.  Suppose that  
$\varphi'=d_1\varphi_1+d_2\varphi_2+\dots + d_k\varphi_k$, 
where $d_1, d_2, \dots, d_k$ 
are complex constants, which are uniquely determined by $\varphi'$.  
  Then, we have 
$$\left|\int_\alpha^\beta e^{i\varphi(t)} \, dt\right|\leq 
C\left(|d_1|+|d_2|+\dots +|d_k|\right)^{-1/k},$$  
where $C$ is independent of $\varphi;$ also the constant $C$ 
 is  independent 
of $\alpha, \beta$ if they are within a fixed finite interval of $\Bbb R$. 
\end{proposition} 
To prove  Proposition 4 we use the following two lemmas.  Both of them are 
well-known.  
\begin{lemma}  
Let $\varphi$ be a solution of \eqref{c2.1}.  
 Suppose that $\varphi$ is not identically $0$.  
Then, there exists a positive integer $K$ independent of $\varphi$ 
such that $\varphi$ has at most 
$K$ zeros in $I$.     
\end{lemma}  
\begin{lemma}[van der Corput]  
Let $f: [c,d]\to \Bbb R$ and $f\in C^j([c,d])$ 
for some positive integer $j$, where $[c,d]$ is an arbitrary compact 
interval in $\Bbb R$. 
Suppose that $\inf_{u\in [c,d]}|(d/du)^jf(u)|\geq \lambda >0$. 
When $j=1$, we further assume that $f^\prime$ is monotone on $[c,d]$. 
Then 
$$\left|\int_c^d e^{if(u)} \, du\right|\leq C_j\lambda^{-1/j},$$  
where $C_j$ is a positive constant depending only on $j$. 
 $($See \cite{SW2, Z}$)$.  
\end{lemma}  
\par 
We now give a proof of Proposition 4.  
We consider linear combinations 
$c_1\varphi_1+c_2\varphi_2+\dots + c_k\varphi_k$, where 
$c_1, c_2, \dots, c_k\in \Bbb C$.  We   
write  $\psi=c_1\varphi_1+c_2\varphi_2+\dots + c_k\varphi_k$ and define 
\begin{gather*}
N_1(\psi)=|c_1|+|c_2|+\dots+|c_k|, 
\\ 
 N_2(\psi) 
=\min_{t\in I}\left(|\psi(t)|+|\psi'(t)|+\dots+|\psi^{(k-1)}(t)|\right).  
\end{gather*} 
Let $U=\{(c_1,c_2,\dots ,c_k)\in \Bbb C^k: |c_1|+|c_2|+\dots+|c_k|=1\}$.  
We consider a function $F$ on $I\times U$ defined by 
$$F(t, c_1,c_2, \dots,c_k)=|\psi(t)|+|\psi'(t)|+\dots+|\psi^{(k-1)}(t)|.$$ 
Then, the function $F$ is continuous and positive on $I\times U$ 
(see \cite{C}).  Thus,  if we put 
$$C_0=\min_{(t, c_1,c_2, \dots,c_k)\in I\times U}F(t, c_1,c_2, \dots,c_k), $$ 
then we see that $C_0>0$ and $N_2(\psi)\geq C_0N_1(\psi)$.  
\par  
Therefore, if $\varphi$ is as in  Proposition 4, we have 
\begin{equation} 
\min_{t\in I}\left(|\varphi'(t)|+|\varphi''(t)|+\dots+|\varphi^{(k)}(t)|\right) \geq C_0N_1(\varphi').                 \label{c2.2} 
\end{equation}    
 By \eqref{c2.2}, for any $t\in I$, there exists         
 $\ell\in \{1,2,\dots , k\}$ such that 
$$|(d/dt)^{\ell}\varphi(t)|\geq CN_1(\varphi'), \quad C>0.$$    
Applying Lemma 3 suitably,  we can decompose $I
=\cup_{m=1}^HI_m$, where $H$ is a positive integer independent 
of $\varphi$ and $\{I_m\}$ is a 
family of non-overlapping subintervals of $I$ such that  
for any interval $I_m$ there is $\ell_m \in \{1,2,\dots , k\}$ 
satisfying 
$|(d/dt)^{\ell_m}\varphi(t)|\geq |(d/dt)^j\varphi(t)|$ on $I_m$ for all 
$j\in \{1,2,\dots , k\}$, so $|(d/dt)^{\ell_m}\varphi(t)|\geq 
CN_1(\varphi')$ on $I_m$,   and such that  
 $\varphi'$ is monotone on each $I_m$.  
Therefore, by Lemma 4 we have 
\begin{align*}  
\left|\int_\alpha^{\beta}e^{i\varphi(t)} \,dt\right|
&=\left|\sum_{m=1}^H\int_{I_m}e^{i\varphi(t)} \, dt\right|
\leq C\sum_{m=1}^H\min\left(|I_m|, N_1(\varphi')^{-1/\ell_m}\right)
\\ 
&\leq CN_1(\varphi')^{-1/k} . 
\end{align*}  
 Since $N_1(\varphi')=|d_1|+|d_2|+\dots + |d_k|$, this completes 
 the proof of  Proposition 4.   
\par 
Proof of Theorem 5.  
By the change of variables $t=e^s$ and an integration by parts argument, 
we can see that to prove  Theorem 5 
it suffices to show 
\begin{equation}
\left|\int_\alpha^\beta \exp\left(i\langle e^{tA}\eta,\zeta\rangle\right)\,dt
\right|\leq CJ(A,\eta,\zeta)^{-1/N}    \label{c2.3} 
\end{equation} 
for an appropriate constant $C>0$, where $[\alpha,\beta]$ is an arbitrary 
compact interval in $\Bbb R$.  
Let $\psi(t)=\langle e^{tA}\eta,\zeta\rangle$.  Then, 
$\psi'(t)=\langle e^{tA}\eta,A^*\zeta\rangle$, and hence, 
 by \eqref{c1.2} we have 
$$\psi'(t)=\sum_{i=1}^{k}
\sum_{j=0}^{m_i-1} c_{ij}(\eta,\zeta)t^je^{\gamma_it}, $$ 
where  
$$c_{ij}(\eta,\zeta)=\frac{1}{j!}\langle (A-\gamma_i E)^jP_i(A)\eta, 
A^*\zeta\rangle. $$  
It is known  that
  $N$ functions $t^je^{\gamma_it}$ ($0\leq j\leq m_i-1$, $1\leq i\leq k$)  
form a basis for the space of solutions for the ordinary differential 
equation of order $N$ with characteristic polynomial $\phi_A$ 
(see \cite{C}).  
Thus, the estimate 
\eqref{c2.3} immediately follows from  Proposition 4, since 
$\sum_{i=1}^k\sum_{j=0}^{m_i-1}\left|c_{ij}(\eta,\zeta)\right|
\approx  J(A,\eta,\zeta)$.


\begin{thebibliography}{99} 
\bibitem{A-SP} A.~Al-Salman and Y.~Pan,  
{\it Singular integrals with rough kernels in $L\log L(S^{n-1})$},  
 J.~London Math.~Soc. (2) {\bf 66} (2002),  153--174.   
\bibitem{CT} A.~P.~Calder\'on and A.~Torchinsky,  
 {\it Parabolic maximal functions associated with a distribution},   
 Advances in Math.  {\bf 16} (1975),  1--64.     
\bibitem{CZ} A.~P.~Calder\'on and A.~Zygmund, 
 {\it On singular integrals},  Amer.~J.~Math. {\bf 78}  
(1956),  289--309.  
 \bibitem{C}  E.~A.~Coddington,   
{\it An Introduction to Ordinary Differential Equations},  Prentice-Hall, 
Inc, Englewood Cliffs, N.~J., 
1961. 
\bibitem{DR} 
 J.~Duoandikoetxea and J.~L.~ Rubio de Francia, 
{\it Maximal and singular integral operators via Fourier transform estimates}, 
 Invent.~Math. {\bf 84}  (1986), 541--561. 
\bibitem{FP}  D.~Fan  and Y.~Pan,  
{\it Singular integral operators with rough kernels supported by 
subvarieties},   Amer.~J.~Math. {\bf 119} (1997),   799--839.      
\bibitem{F} R.~Fefferman,  
 {\it A note on singular integrals},  
Proc.~Amer.~Math.~Soc. {\bf 74} (1979),  266--270.   
 \bibitem{FIS}  S.~Friedberg, A.~Insel and L.~Spence,   
{\it Linear Algebra},  Prentice-Hall, Inc, Englewood Cliffs, N.~J., 
1979. 
\bibitem{H} S.~Hofmann,  {\it Weighted norm inequalities and vector valued 
inequalities for certain rough operators}, Indiana Univ.~Math.~J. 
 {\bf 42} (1993), 1--14.     
\bibitem{KWWZ}  W.~Kim, S.~Wainger, J.~Wright and S.~Ziesler,   
 {\it Singular integrals and maximal functions associated to surfaces 
of revolution}, 
 Bull.~London Math.~Soc.  {\bf 28} (1996),   291--296.   
 \bibitem{NW}  A.~Nagel and S.~Wainger,   
{\it $L^2$ boundedness of Hilbert transforms along surfaces and convolution 
operators homogeneous with respect to a multiple parameter group},  
 Amer.~J.~Math.   {\bf 99} (1977),  761--785.   
\bibitem{N} J.~Namazi, 
{\it On a singular integral},  
 Proc.~Amer.~Math. Soc.  {\bf 96} (1986),  421--424.  
\bibitem{R} N.~Rivi\`{e}re, 
{\it Singular integrals and multiplier operators},  
 Ark.~Mat.  {\bf 9} (1971),    243--278.         
\bibitem{S} S.~Sato, 
{\it Estimates for singular integrals and extrapolation},   
arXiv:0704.1537v1 [math.CA].          
\bibitem{St}  E.~M.~Stein,      
{\it Harmonic Analysis$:$  Real-Variable Methods, Orthogonality and 
Oscillatory Integrals},    Princeton University Press   Princeton, NJ, 1993.   
\bibitem{SW}  E.~M.~Stein and S.~Wainger,  
{\it The estimation of an integral arising in multiplier transformations},   
 Studia Math.   {\bf 35} (1970),    101--104.    
\bibitem{SW2}  E.~M.~Stein and S.~Wainger,  
 {\it Problems in harmonic analysis related to curvature},   
 Bull.~Amer.~Math.~Soc.   {\bf 84}  (1978),    1239--1295.    
\bibitem{Z}  A.~Zygmund,  
{\it Trigonometric Series},  
 2nd ed.,   Cambridge Univ. Press,  
 Cambridge, London, New York and Melbourne, 1977.  

\end{thebibliography}
\end{document}